\documentclass[11pt]{article}

\usepackage[english]{babel}
\usepackage{hyperref}
\usepackage{setspace,color}
\usepackage{graphicx,boxedminipage}
\usepackage{url}
\usepackage{amsmath,amssymb}
\usepackage{booktabs}
\usepackage{lineno}
\usepackage{xcolor}
\usepackage{bbm}
\usepackage{float}
\usepackage{subcaption}
\usepackage{tikz}
\usepackage[margin=1in]{geometry}
\usepackage[T1]{fontenc}
\usepackage{lmodern}
\usepackage{amsthm}

\newtheorem{Theorem}{Theorem}[section]
\newtheorem{Lemma}[Theorem]{Lemma}

\newtheorem{Remark}[Theorem]{Remark}

\numberwithin{equation}{section}

\newcommand{\eee}{{\rm e}}
\newcommand{\ddd}{{\rm d}}
\allowdisplaybreaks
\newcommand{\X}{\mathbf{X}}

\renewcommand{\P}{\mathbb{P}}
\newcommand{\E}{\mathbb{E}}
\newcommand{\dd}{\mathrm{d}}

\title{Optimal Structure of Signal Networks for\\ Efficient Information Aggregation}
\date{}

\author{%
\begin{minipage}{\textwidth}\centering
Bernd Heidergott$^{1}$, Frank den Hollander$^{2}$, Ines Lindner$^{3*}$, Azadeh Parvaneh$^{2}$\\[0.5em]
\small
$^{1}$ VU University Amsterdam, Department of Operations Analytics, De Boelelaan 1105, 1081 HV Amsterdam, The Netherlands\\
$^{2}$ Mathematical Institute, Leiden University, Einsteinweg 55, 2333 CC Leiden, The Netherlands\\
$^{3}$ VU University Amsterdam, Department of Economics, De Boelelaan 1105, 1081 HV Amsterdam, The Netherlands\\[0.5em]
$^{*}$Corresponding author: \texttt{i.d.lindner@vu.nl}
\end{minipage}
}

\begin{document}

\maketitle

\begin{abstract}
This paper develops a mathematical framework to study signal networks, in which nodes can be active or inactive, and their activation or deactivation is driven by external signals and the states of the nodes to which they are connected via links. The focus is on determining the optimal number of key nodes (= highly connected and structurally important nodes) required to represent the global activation state of the network accurately. Motivated by neuroscience, medical science, and social science examples, we describe the node dynamics as a continuous-time inhomogeneous Markov process. Under mean-field and homogeneity assumptions, appropriate for large scale-free and disassortative signal networks, we derive differential equations characterising the global activation behaviour and compute the expected hitting time to network triggering. Analytical and numerical results show that two or three key nodes are typically sufficient to approximate the overall network state well, balancing sensitivity and robustness. Our findings provide insight into how natural systems can efficiently aggregate information by exploiting minimal structural components.

\emph{Acknowledgment}: FdH and AP were supported by the Netherlands Organisation for Scientific Research (NWO) through NETWORKS Gravitation Grant no.\ 024.002.003. AP was also supported by the European Union's Horizon 2020 research and innovation programme under the Marie Sk\l odowska-Curie grant agreement no.\ 101034253.

\vspace{0.3cm}\hspace{-0.3cm}
\includegraphics[height=3em]{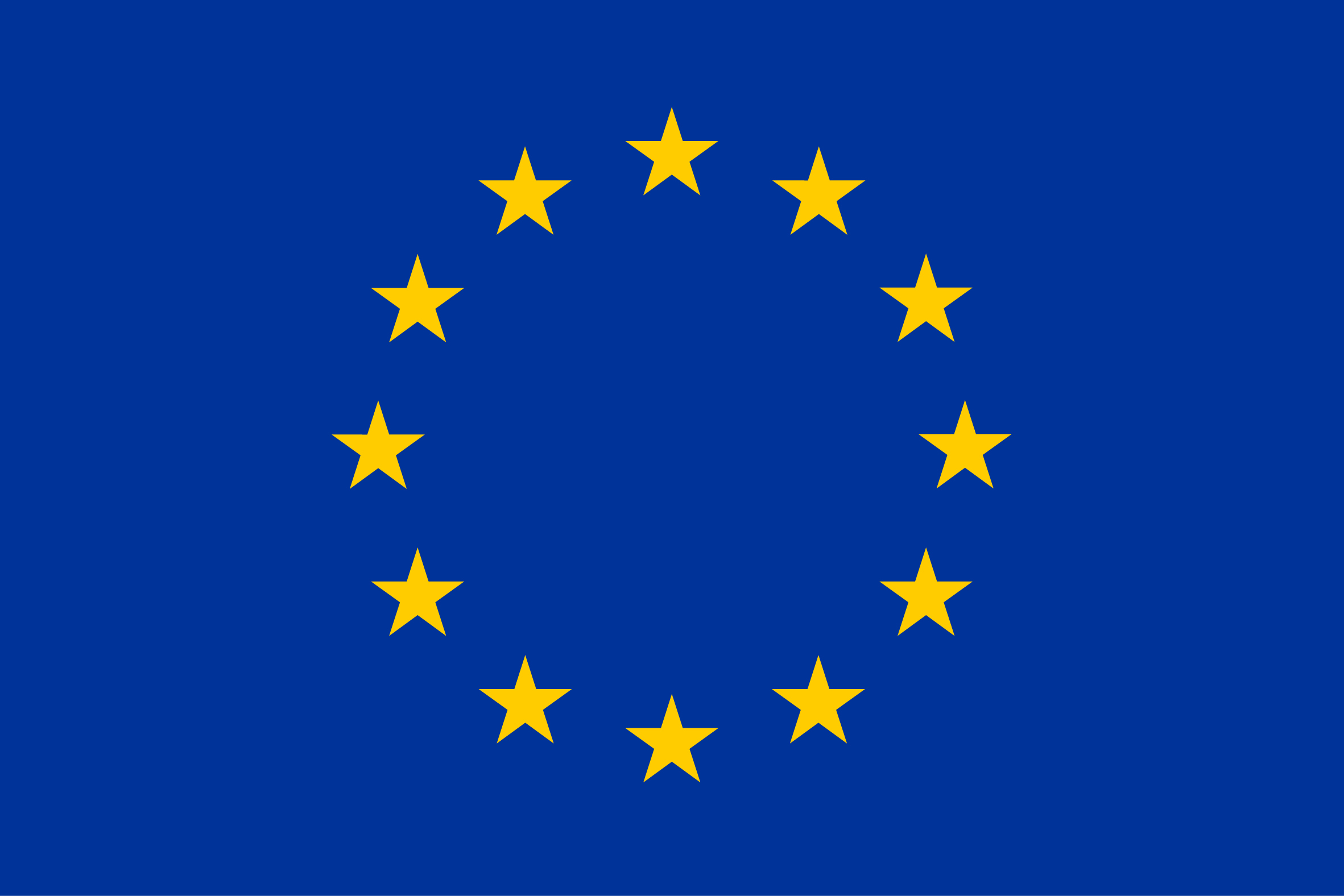}

\end{abstract}

\maketitle


\section{Introduction}

This paper investigates how to efficiently aggregate nodes in a signal network to accurately and concisely represent the network's overall state. Our analysis is motivated by the observation that, in many cases, a very small number of nodes suffices to capture the essential behaviour of the entire network, even when the network itself is large. Section~\ref{ss.SN} introduces the concept of signal networks and outlines the study's main objectives. Sections~\ref{ss.BN} through \ref{ss.FN} illustrates our framework's relevance through three key examples: brain networks, aging networks, and friendship networks.


\subsection{Signal networks}\label{ss.SN}

Consider a network with a large number of nodes. The state of a node is modeled with the help of a strictly decreasing {\em activation function} $\phi (u)$, $u \geq 0$, as follows. If at time $t$ a node is triggered by an exogenous signal, then its state at time $t' > t$ is $\phi (t' - t)$. As long as the state is larger than or equal to a threshold value, the node is {\em active}, while as soon as the state drops below the threshold value, the node is {\em inactive}. The prior information gets overridden when the node is triggered again at a later point in time, and the triggering process restarts independently of the past.

We represent large sets of nodes as an {\em aggregated node}. We assume that the time instances of signals at individual nodes form a general renewal process. Under this mild assumption, the process of active versus inactive states for the aggregated nodes is well approximated by a {\em Poisson point process} with exponentially distributed holding times between active states due to the {\em superposition principle} for a large number of sparse point processes \cite[Theorem 3.1]{Aging}. Our primary focus will be on the subnetwork formed by the aggregated nodes, which we refer to as the {\em signal network}. The network is encoded as `triggered' when a fraction at least $\gamma \in (0,1)$ of the aggregated nodes is active. The time-lapse during which an aggregated node stays active or inactive depends on the state of its neighboring aggregated nodes. We assume that the signal network is \textit{scale-free}, meaning that its empirical degree distribution is approximately a power law, and \textit{disassortative}, meaning that node degrees tend to be negatively correlated. We assume that the signal network contains $n \gg 1$ aggregated nodes, of which $1 \leq k \ll n$ are {\em key nodes}, i.e., `hubs' in the signal network that play a critical role in the functioning of the network, in the sense that when they are active the entire network is triggered. See Fig.\ \ref{fig:network} for an illustration.
 
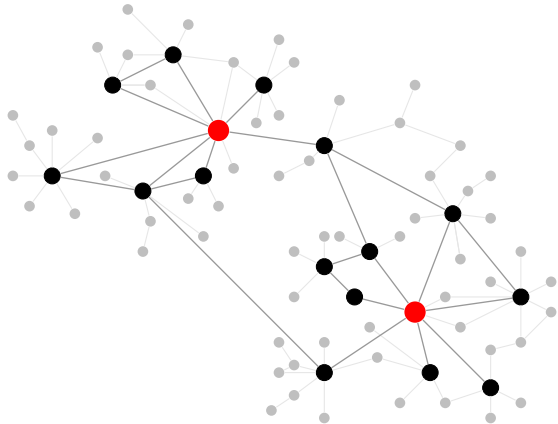
\begin{figure}[htbp]
\centering
\begin{tikzpicture}[scale=2, every node/.style={circle, draw, minimum size=4pt, inner sep=1pt}]    
    \node[fill=red, draw=none, minimum size=8pt] (A) at (0, 1.5) {}; 
    \node[fill=red, draw=none, minimum size=8pt] (B) at (1.3, 0.3) {}; 
    \node[fill=black, minimum size=6pt] (0) at (0.7, 1.4) {};
   \node[fill=black, minimum size=6pt] (1) at (-0.3, 2) {};
     \node[fill=black, minimum size=6pt] (2) at (0.3, 1.8) {};
   \node[fill=black, minimum size=6pt] (3) at (-0.7, 1.8) {};
     \node[fill=black, minimum size=6pt] (4) at (-1.1, 1.2) {};
     \node[fill=black, minimum size=6pt] (5) at (-0.5, 1.1) {};
        \node[fill=black, minimum size=6pt](5b) at (-0.1, 1.2) {};
   \node[fill=black, minimum size=6pt] (6) at (1, 0.7) {};
         \node[fill=black, minimum size=6pt] (6b) at (1.55, 0.95) {};
  \node[fill=black, minimum size=6pt](7) at (0.7, 0.6) {};
      \node[fill=black, minimum size=6pt] (8) at (0.9, 0.4) {};
     \node[fill=black, minimum size=6pt] (9) at (2.0, 0.4) {};
    \node[fill=black, minimum size=6pt](10) at (1.8, -0.2) {};
     \node[fill=black, minimum size=6pt](11) at (1.4, -0.1) {};
    \node[fill=black, minimum size=6pt] (12) at (0.7, -0.1) {};
        \node[fill=gray!50, draw=none, minimum size=4pt] (13) at (0.7, -0.4) {};
         \node[fill=gray!50, draw=none, minimum size=4pt] (14) at (0.35, -0.35) {};
\node[fill=gray!50, draw=none, minimum size=4pt] (15) at (0.5, -0.25) {};
  \node[fill=gray!50, draw=none, minimum size=4pt] (16) at (0.4, -0.1) {};
  \node[fill=gray!50, draw=none, minimum size=4pt] (17) at (0.5, -0.05) {};
 \node[fill=gray!50, draw=none, minimum size=4pt] (18) at (0.4, 0.1) {};
 \node[fill=gray!50, draw=none, minimum size=4pt] (19) at (0.7, 0.1) {};
  \node[fill=gray!50, draw=none, minimum size=4pt] (20) at (1.05, 0) {};
\node[fill=gray!50, draw=none, minimum size=4pt] (21) at (1, 0.2) {};
\node[fill=gray!50, draw=none, minimum size=4pt] (22) at (1.2, -0.3) {};
\node[fill=gray!50, draw=none, minimum size=4pt] (23) at (1.5, -0.3) {};
 \node[fill=gray!50, draw=none, minimum size=4pt] (24) at (1.8, -0.4){};
 \node[fill=gray!50, draw=none, minimum size=4pt] (25) at (2, -0.3) {};
\node[fill=gray!50, draw=none, minimum size=4pt] (26) at (1.8, 0.05){};
 \node[fill=gray!50, draw=none, minimum size=4pt] (27) at (2, 0.1){};
\node[fill=gray!50, draw=none, minimum size=4pt] (28) at (1.6, 0.2){};
 \node[fill=gray!50, draw=none, minimum size=4pt] (29) at (2.2, 0.3){};
 \node[fill=gray!50, draw=none, minimum size=4pt] (30) at (2.2, 0.5){};
 \node[fill=gray!50, draw=none, minimum size=4pt] (31) at (2, 0.7){};
\node[fill=gray!50, draw=none, minimum size=4pt] (32) at (1.8, 0.5){};
\node[fill=gray!50, draw=none, minimum size=4pt] (33) at (1.5, 0.4){};
\node[fill=gray!50, draw=none, minimum size=4pt] (35) at (1.3, 0.92) {};
\node[fill=gray!50, draw=none, minimum size=4pt] (36) at (1.6, 0.65) {};
\node[fill=gray!50, draw=none, minimum size=4pt] (37) at (1.8, 0.92) {};
\node[fill=gray!50, draw=none, minimum size=4pt] (38) at (1.65, 1.1) {};
\node[fill=gray!50, draw=none, minimum size=4pt] (39) at (1.8, 1.2) {};
\node[fill=gray!50, draw=none, minimum size=4pt] (40) at (1.4, 1.2) {};
\node[fill=gray!50, draw=none, minimum size=4pt] (41) at (1.6, 1.4) {};
 \node[fill=gray!50, draw=none, minimum size=4pt] (42) at (1.2, 1.55) {};
 \node[fill=gray!50, draw=none, minimum size=4pt] (43) at (1.2, 0.8) {};
\node[fill=gray!50, draw=none, minimum size=4pt] (44) at (0.8, 0.8) {};
 \node[fill=gray!50, draw=none, minimum size=4pt] (45) at (0.5, 0.7) {};
 \node[fill=gray!50, draw=none, minimum size=4pt] (46) at (0.5, 0.4) {};
\node[fill=gray!50, draw=none, minimum size=4pt] (47) at (0.7, 0.8) {};
\node[fill=gray!50, draw=none, minimum size=4pt] (48) at (1.3, 1.8) {};
\node[fill=gray!50, draw=none, minimum size=4pt] (49) at (0.8, 1.7) {};
\node[fill=gray!50, draw=none, minimum size=4pt] (50) at (0.6, 1.3) {};
\node[fill=gray!50, draw=none, minimum size=4pt] (51) at (0.4, 1.2) {};
\node[fill=gray!50, draw=none, minimum size=4pt] (53) at (0.4, 1.6) {};
\node[fill=gray!50, draw=none, minimum size=4pt] (54) at (0.25, 1.55) {};
\node[fill=gray!50, draw=none, minimum size=4pt] (55) at (0.5, 1.95) {};
\node[fill=gray!50, draw=none, minimum size=4pt] (56) at (0.4, 2.1) {};
\node[fill=gray!50, draw=none, minimum size=4pt] (57) at (0.1, 1.95) {};
\node[fill=gray!50, draw=none, minimum size=4pt] (58) at  (-0.2, 2.2) {};
 \node[fill=gray!50, draw=none, minimum size=4pt] (59) at  (-0.6, 2.3) {};
 \node[fill=gray!50, draw=none, minimum size=4pt] (60) at  (-0.6, 2.0) {};
\node[fill=gray!50, draw=none, minimum size=4pt] (61) at  (-0.8, 2.05) {};
 \node[fill=gray!50, draw=none, minimum size=4pt] (62) at  (-0.45, 1.8) {};
 \node[fill=gray!50, draw=none, minimum size=4pt] (63) at  (0.1, 1.25) {};
 \node[fill=gray!50, draw=none, minimum size=4pt] (64) at   (-0.8, 1.45) {};
 \node[fill=gray!50, draw=none, minimum size=4pt] (65) at   (-1.1, 1.5) {};
\node[fill=gray!50, draw=none, minimum size=4pt] (66) at   (-1.25, 1.4) {};
\node[fill=gray!50, draw=none, minimum size=4pt] (67) at   (-1.25, 1) {};
\node[fill=gray!50, draw=none, minimum size=4pt] (68) at   (-1.36, 1.2) {};
\node[fill=gray!50, draw=none, minimum size=4pt] (69) at   (-1.36, 1.6) {};
 \node[fill=gray!50, draw=none, minimum size=4pt] (70) at   (-0.95, 0.95) {};
\node[fill=gray!50, draw=none, minimum size=4pt] (71) at  (-0.75, 1.2) {};
 \node[fill=gray!50, draw=none, minimum size=4pt] (72) at  (-0.45, 0.9) {};
\node[fill=gray!50, draw=none, minimum size=4pt] (73) at  (-0.5, 0.7) {};
 \node[fill=gray!50, draw=none, minimum size=4pt] (74) at  (-0.2, 1.05) {};
\node[fill=gray!50, draw=none, minimum size=4pt] (75) at  (-0, 1) {};
\node[fill=gray!50, draw=none, minimum size=4pt] (76) at  (-0.1, 0.8) {};
\draw (12) edge[gray!20, thin] (13);
 \draw (12) edge[gray!20, thin] (15);
 \draw (14) edge[gray!20, thin] (15);
  \draw (16) edge[gray!20, thin] (12);
   \draw (17) edge[gray!20, thin] (18);
  \draw (12) edge[gray!20, thin] (17);
  \draw (16) edge[gray!20, thin] (17);
   \draw (12) edge[gray!20, thin] (19);
  \draw (12) edge[gray!20, thin] (20);
   \draw (11) edge[gray!20, thin] (20);
 \draw (11) edge[gray!20, thin] (21);
 \draw (11) edge[gray!20, thin] (22);
   \draw (11) edge[gray!20, thin] (23);
  \draw (10) edge[gray!20, thin] (23);
  \draw (10) edge[gray!20, thin] (24);
       \draw (10) edge[gray!20, thin] (25);
  \draw (10) edge[gray!20, thin] (26);
 \draw (27) edge[gray!20, thin] (26);
 \draw (27) edge[gray!20, thin] (9);
    \draw (9) edge[gray!20, thin] (28);
     \draw (B) edge[gray!20, thin] (28);
 \draw (9) edge[gray!20, thin] (29);
 \draw (27) edge[gray!20, thin] (29);
  \draw (9) edge[gray!20, thin] (30);
 \draw (9) edge[gray!20, thin] (31);
 \draw (9) edge[gray!20, thin] (32);
 \draw (9) edge[gray!20, thin] (33);
   \draw (B) edge[gray!20, thin] (33);
 \draw (6b) edge[gray!20, thin] (36);
  \draw (6b) edge[gray!20, thin] (35);
\draw (6b) edge[gray!20, thin] (36);
\draw (6b) edge[gray!20, thin] (37);
   \draw (6b) edge[gray!20, thin] (38);
   \draw (38) edge[gray!20, thin] (39);
  \draw (6b) edge[gray!20, thin] (40);
   \draw (40) edge[gray!20, thin] (41);
   \draw (42) edge[gray!20, thin] (41);
 \draw (6) edge[gray!20, thin] (43);
 \draw (6) edge[gray!20, thin] (44);
  \draw (7) edge[gray!20, thin] (45);
   \draw (7) edge[gray!20, thin] (47);
  \draw (7) edge[gray!20, thin] (46);
 \draw (42) edge[gray!20, thin] (48);
 \draw (42) edge[gray!20, thin] (0);
\draw (49) edge[gray!20, thin] (0);
 \draw (50) edge[gray!20, thin] (0);
 \draw (50) edge[gray!20, thin] (51);
 \draw (2) edge[gray!20, thin] (57);
    \draw (2) edge[gray!20, thin] (56);
\draw (2) edge[gray!20, thin] (55);
\draw (2) edge[gray!20, thin] (54);
\draw (2) edge[gray!20, thin] (53);
\draw (A) edge[gray!20, thin] (57);
\draw (1) edge[gray!20, thin] (59);
\draw (1) edge[gray!20, thin] (58);
\draw (1) edge[gray!20, thin] (57);
\draw (1) edge[gray!20, thin] (60);
 \draw (3) edge[gray!20, thin] (60);
 \draw (3) edge[gray!20, thin] (61);
\draw (3) edge[gray!20, thin] (62);
\draw (3) edge[gray!20, thin] (62);
\draw (A) edge[gray!20, thin] (63);
  \draw (A) edge[gray!20, thin] (62);
   \draw (4) edge[gray!20, thin] (70);
 \draw (4) edge[gray!20, thin] (68);
\draw (4) edge[gray!20, thin] (67);
 \draw (4) edge[gray!20, thin] (66);
\draw (4) edge[gray!20, thin] (65);                                                                                                                                            \draw (4) edge[gray!20, thin] (64);                                                                                                                                            \draw (69) edge[gray!20, thin] (66);                                                                                                                                              \draw (5) edge[gray!20, thin] (71);                                                                                                                                                    \draw (5) edge[gray!20, thin] (72);                                                                                                                                                      \draw (73) edge[gray!20, thin] (72);                                                                                                                                                       \draw (5b) edge[gray!20, thin] (74);                                                                                                                                                        \draw (5b) edge[gray!20, thin] (75);                                                                                                                                                         \draw (5) edge[gray!20, thin] (76);
        \draw[gray!80, thick, line width=0.5pt] (A) -- (5b);
        \draw[gray!80, thick, line width=0.5pt] (5b) -- (5);
        \draw[gray!80, thick, line width=0.5pt] (0) -- (A);
    \draw[gray!80, thick, line width=0.5pt] (A) -- (5);
    \draw[gray!80, thick, line width=0.5pt] (A) -- (1);
    \draw[gray!80, thick, line width=0.5pt] (A) -- (2);
    \draw[gray!80, thick, line width=0.5pt] (A) -- (3);
    \draw[gray!80, thick, line width=0.5pt] (A) -- (4);
   \draw[gray!80, thick, line width=0.5pt] (7) -- (6);
    \draw[gray!80, thick, line width=0.5pt] (1) -- (3);
    \draw[gray!80, thick, line width=0.5pt] (4) -- (5);
   \draw[gray!80, thick, line width=0.5pt] (7) -- (8);
   \draw[gray!80, thick, line width=0.5pt] (0) -- (6);
   \draw[gray!80, thick, line width=0.5pt] (6) -- (B);
    \draw[gray!80, thick, line width=0.5pt] (B) -- (8);
    \draw[gray!80, thick, line width=0.5pt] (B) -- (9);
    \draw[gray!80, thick, line width=0.5pt] (B) -- (10);
    \draw[gray!80, thick, line width=0.5pt] (B) -- (11);
    \draw[gray!80, thick, line width=0.5pt] (B) -- (12);
  \draw[gray!80, thick, line width=0.5pt] (5) -- (12);
  \draw[gray!80, thick, line width=0.5pt] (6b) -- (B);
  \draw[gray!80, thick, line width=0.5pt] (6b) -- (0);
  \draw[gray!80, thick, line width=0.5pt] (6b) -- (9);
\end{tikzpicture}
\caption{\small An example of a signal network \cite{Aging}. The aggregated nodes are black and the key nodes are red. The grey nodes form the background of the aggregated nodes.}
\label{fig:network}
\end{figure}

Think of the bodily reaction of creating `goosebumps' after an emotional experience. For this reaction to occur, the network of aggregated nodes has to be triggered across a certain threshold. Our claim is that this triggering is best organised by aggregating the overall state of the network into a monitoring of the state of the key nodes. Our main question is: What is the optimal value for $k$ to do this efficiently?
\begin{itemize}
\item 
\textbf{Claim:} If the network has $k=2$ or $3$ key nodes, then the states of these key nodes capture the overall state of the network sufficiently well. In contrast, for $k=1$, the key node may already be active when far less than $\gamma$ of the aggregated nodes are active, while for $k \geq 4$, the key nodes may still be inactive when far more than $\gamma$ of the aggregated nodes are active.
\end{itemize} 
In other words, it suffices to take a \emph{snapshot} of the activity state of a very small number of key nodes to get a fair impression of the \emph{overall} activity state of a large signal network. (In order for this claim to be valid, the signal network must avoid absorbing states, for which the distributions of the active and inactive times have to be chosen appropriately: the more neighbours are active, the shorter the time of inactivity and the longer is the time of activity.) The idea behind this claim is that, in a large scale-free and disassortative network, the two or three nodes with the highest degree and centrality properly represent the state of most of the network with minimal overlap, so that adding more key nodes provides little extra coverage and only delays triggering. These structural properties also make the mean-field and homogeneity assumptions reasonable (Sections \ref{ss.MF} and \ref{ss.H}).

Several recent studies have considered network robustness from different perspectives \cite{new-ref1, new-ref2, new-ref3, new-ref4}. In contrast to previous studies on network robustness, which examine the ability of a network to maintain functionality under node or link failures or targeted attacks, our work focuses on identifying the minimal number of key nodes whose state accurately represents the global state of the network. We develop a rigorous mathematical framework, rather than relying solely on empirical data or simulations, and show that in large scale-free and disassortative networks only a very small number of key nodes is needed to represent the overall network state. 

Our work is inspired by the study in \cite{MRFR2017}, where empirical data suggested that two mortality nodes are sufficient to represent the state of the network. Later, in \cite{Aging}, this number of mortality nodes was adopted by referring to \cite{MRFR2017}. The goal of the present paper is to explain, mathematically and rigorously, why such a small number of key nodes works, without relying on empirical data, which can be misleading when the data are not sampled correctly. We generalise the framework, so that it applies not only to the networks studied in \cite{MRFR2017, Aging}, but also to a broader class of networks with a similar underlying structure, such as brain networks and friendship networks. The numerical results in Section~\ref{s.numerics} are obtained solely by solving differential equations that follow directly from the theory.


\subsection{Brain network as signal network}
\label{ss.BN}

A brain network in cognitive neuroscience is an important example of a signal network that reflects how the brain processes information. A brain network can be described as a graph, with nodes representing neural elements (e.g., \ neurons or brain regions) and links representing anatomical connectivity (e.g., \ synapses or axonal projections). Such a network is called a {\em structural brain network}. If the links describe dynamic interactions or statistical dependencies between the nodes, it is called a {\em functional brain network.} Signals received by the body are transmitted to the brain directly via cranial nerves or the spinal cord. Unlike graph theory, which emphasises only connectivity patterns, structural brain networks are influenced by physical and topological distances. Neurons and brain regions that are spatially close are more likely to be connected, whereas neurons or brain regions that are spatially distant are less likely to be connected \cite{BullmoreSporns2009}. In \cite{Sporns2011}, it is argued that the key nodes of a brain network are highly connected and central (the `hubs' in the network) and that these nodes serve as primary points for the overall functioning of the network. Several empirical studies have shown that \emph{small-world} architectures are present in both structural and functional brain networks, in humans and other animals, across a wide range of space and time scales \cite{SW1, SW2, SW3, SW4, BullmoreSporns2009}. The small-world architecture ensures short distances for quick signal processing, while \emph{sufficient clustering} ensures stability (when a node falls out). In addition, the brain has a strong synchronisation capacity, which the brain modules need in order to be functional, similar to the firefly synchronisation phenomenon \cite{Buchanan}.

Regarding the \emph{scale-freeness} of brain networks, the reports differ depending on the data sets and the approach taken. Due to physical constraints and the cost of adding connections in the brain, there are strict upper limits on the number and the density of connections at any given node. As a result, structural brain networks (even those at a large scale) are unlikely to exhibit scale-free degree distributions across a broad range of node degrees \cite{Amaral2000, Sporns2011}. However, an analysis of the degree distribution reveals deviations from a Gaussian or an exponential profile \cite{Sporns2011}. Truncated power-law degree distributions have also been reported for humans \cite{55}.

Functional brain networks show different degree distributions depending on the scale of observation. Scale-freeness evidence has been reported in voxel-level analyses \cite{69,91}, while exponentially truncated power-law distributions have been reported in region-level analyses \cite{70,81}. Moreover, truncated power-law degree distributions have been reported for cats \cite{86}. Disassortativity is often considered a natural characteristic of biological networks \cite{Newman2003}. However, studies show that the structural brain network exhibits an \emph{assortative degree organisation} \cite{Hagmann2008, Braun2012}, while the functional brain network is disassortative \cite{Bettencourt2007}. These findings are further supported by \cite{Lim2019}. Disassortativity helps to ensure that if one of the key nodes is damaged, then the impact is restricted to only a small part of the network, which is crucial for overall stability and functionality.

The \emph{centrality of hubs} means that they are often part of the shortest paths between other nodes, which makes their role important. Hubs play a crucial role in the brain by connecting different signals and controlling the flow of information between separate areas in the brain. Because much of the communication between brain regions passes through these hubs, their function has a significant impact on the overall brain performance. Any perturbation in the state of a hub may quickly spread throughout the network. Moreover, hubs help conserve wiring length and volume because they allow information to travel efficiently without needing long-range connections. As a result, if hubs are damaged or malfunctioning, this can substantially impact the overall function of the brain \cite{Sporns2011}. Understanding the optimal number of hubs in a brain network is therefore crucial, as it provides valuable insight into the overall brain function. Our main interest will be functional brain networks.


\subsection{Aging network as signal network}
\label{ss.AN}

The body can be seen as a network, where nodes represent functional units (organs, cells, or health measurements) and links represent interactions or dependencies between them. Aging can be viewed as the accumulation of damage across this network over time. Damage can spread when one part of the system fails due to internal degradation or external signals. This failure increases the risk of damage in neighboring parts. Some nodes, such as vital organs in the human body, are critical for the overall function, and we can, therefore, call them the key nodes of the aging network. When these critical nodes fail, the entire system collapses, leading to death. In 1825, Benjamin Gompertz \cite{G1825} stated an empirical relationship for the mortality rate $m(t)$ at time $t$, given by 
\begin{equation*}
m(t) \approx \alpha\, \mathrm{e}^{\beta t}, \qquad t \geq 0,
\end{equation*}
with parameters $\alpha,\beta>0$ (and with $\approx$ meaning approximately). This relationship is since referred to as \emph{Gompertz law}. It shows that, as we age, our death risk increases exponentially. The Gompertz law works relatively well for adult ages, roughly from $40$ to $90$ years, but not at very young or very old ages.

\cite{Aging} offers a causal mathematical model that explains aging and mortality through network theory, which offers a novel and mathematical derivation of the Gompertz law. In this network, two nodes are designated as \emph{mortality nodes} (most vital nodes), while the others are designated \emph{aging nodes}. The network is assumed to be scale-free and disassortative. Each node has a state that is either healthy or damaged. The health network evolves over time via a Markov process, where nodes can switch between healthy and damaged. Damage to nodes spreads via the links between them. An individual is considered dead when both mortality nodes become damaged. With the help of a mean-field assumption and a homogeneity assumption, the mortality rate is shown to be the product of the solutions of two non-linear differential equations, and this product approximates the Gompertz law in the age range of $40$ to $80$ years.

The aging network studied in \cite{Aging} closely aligns with the signal network considered in the present paper, where active nodes correspond to damaged nodes and passive nodes correspond to healthy nodes. In \cite{Aging}, the number of mortality nodes is set to two, following the network description of mortality proposed in \cite{MRFR2017}, which is based on a suitable empirical fit to data. This choice lacks a theoretical justification, and the preference for two mortality nodes over other values is primarily based on empirical performance. Investigating the optimal number of key nodes, as is done in the present paper, provides a deeper insight into the structural and functional robustness of aging networks.

\subsection{Friendship network as signal network}
\label{ss.FN}

Another example of a signal network is the friendship network, where nodes represent individuals in the community under consideration, and links represent their (mutual) friendships. In this context, signals can be understood as the spread of information, such as news, rumours, or opinions. A node is considered passive when it is not affected by the incoming information and active when it is and can respond to it.

Empirically, social networks often follow a scale-free degree distribution. A widely used model to describe such networks is the preferential attachment model in \cite{BA1999}. This dynamic random network model assumes that new nodes are more likely to attach to already highly connected nodes (i.e., nodes with a high degree), which leads to negative degree correlation and a scale-free topology \cite{BA1999}. While the classical preferential attachment model tends to produce disassortative networks based purely on degree, \cite{Jackson} proposes more socially grounded mechanisms, which can result in assortative network structures where individuals tend to connect with other individuals similar to themselves. However, under certain conditions, particularly when relationships are hierarchical or are driven by influence rather than mutual friendships, the network may also exhibit disassortative characteristics. Examples are social media networks (like X-follower networks) or advisor-student relationships. Moreover, \cite{PSVV, BBV} find that disassortativity is a result of an environment in which the growth mechanism depends on the fitness of a node in competitive dynamics and weight-driven dynamics, respectively. 

The analogy to the body evolving from single-celled organisms, incrementally adding body parts based on fitness criteria, is tempting. With an environment that is changing in time, different body parts and functionalities become less or more important, and therefore, the fitness of a node may change with time. During an ice age, for example, functionalities that keep the body warm would increase fitness, and so, as the body evolves, new nodes that are added will likely attach to nodes that aid survival in the cold, such as fat cells. The result is a \emph{hierarchical organisation}, meaning the occurrence of many small densely connected clusters, which combine to form larger, less densely connected groups, which again combine to form even larger and even less densely connected groups. Our body also exhibits hierarchical organisation with cells making up tissue, tissues making up organs, organs making up organ systems, and organ systems ultimately making up higher-level functionality such as the ability to read. The analogy with the social context is that the body needs to be efficiently organised.

With the above considerations, our assumption of disassortativity in signal networks is \emph{purely degree-based}. Therefore, a friendship network can also be a good signal network example. This idea is further supported by the well-known friendship paradox, which states that, on average, our friends have more friends than we do. A mathematical analysis of the friendship paradox \cite{HHP} shows that only a small fraction of the nodes may have more connections than their friends do in large preferential attachment networks, such as worldwide friendship networks. Finding the optimal number of key nodes (i.e., highly connected and influential individuals) in such networks is an interesting direction for future research in this context.


\section{Model Framework}
\label{s.modass}

Section~\ref{ss.model} formulates the model. Sections~\ref{ss.MF}--\ref{ss.H} state two \emph{key assumptions} under which the model can be analysed in closed form.   


\subsection{Model}
\label{ss.model}

Recall Figure~\ref{fig:network}. Labelling the \emph{aggregated nodes} by $1,\ldots ,n$, we describe the state of the signal network as  a continuous-time stochastic process
\[
\X(t) = (X_{1}(t),\ldots ,X_{n}(t)), \qquad t\geq 0,
\]
where $X_{i}(t) = 1$ when node $i$ is active at time $t$ and $X_{i}(t) = 0$ when node $i$ is passive at time $t$. We let $\X(0) = (0, \ldots, 0)$. The stochastic process $\X=(\X(t))_{t \geq 0}$ is a time-inhomogeneous Markov chain. 

We denote the set of neighbours of node $i$ by $N(i)$. For each node $i$, we assume that the set $N(i)$ is non-empty. We assume that the rates at time $t$ at which node $i$ transitions to the active state or the passive state, respectively, are given by
\begin{align*}
\Gamma_{+}(i,t) = \lambda(a_{i}(t)), \qquad  \Gamma_{-}(i,t)=\mu(a_{i}(t)),
\end{align*}
where $\lambda(\cdot)$ and $\mu(\cdot)$ are decreasing, repectively, increasing functions, and $a_{i}(t)$ is the fraction of active neighbours of node $i$ at time $t$, i.e.,
\begin{align*}
a_{i}(t) = \frac{1}{|N(i)|}\sum_{j\in N(i)}\mathbbm{1}_{\{X_{j}(t)=1\}}.
\end{align*}
We define the average fraction of active nodes at time $t$ by
\begin{align*}
a(t) = \frac{1}{n}\sum_{j=1}^{n}\mathbbm{1}_{\{X_{j}(t)=1\}}.
\end{align*}

Labelling the \emph{key nodes} by $1,\ldots,k$, we define
\begin{align*}
\tau^\gamma = \inf\big\{t\geq 0\colon\, a(t)\geq \gamma\big\}
\end{align*}
to be the first time when the signal network gets triggered and
\begin{align*}
\tau_k = \inf\big\{t\geq 0\colon\, (X_{1}(t),\ldots ,X_k(t))\in (1,\ldots ,1)\big\}
\end{align*}
be the first time when all $k$ key nodes are active. Due to the central role played by the key nodes, the activation of all key nodes should indicate that most aggregated nodes are active (i.e., the system is triggered). Therefore, we should expect that $\mathbb{E}[\tau^\gamma] \approx \mathbb{E}[\tau_k]$ ($\approx$ means approximately) for suitable values of $k$. For this purpose, we define
\begin{align*}
k_c^-(\gamma) &= \sup\big\{1\leq k\leq n\colon\,\mathbb{E}[\tau^\gamma] \geq \mathbb{E}[\tau_k]\big\}\vee 1,\\
k_c^+(\gamma) &=\inf\big\{1\leq k\leq n\colon\,\mathbb{E}[\tau^\gamma] \leq \mathbb{E}[\tau_k]\big\}\wedge n,
\end{align*}
where we use the convention that $\sup\emptyset=-\infty$ and $\inf\emptyset=+\infty$. Since $\tau_k$ is strictly increasing in $k$, the optimal value of $k$ is defined to be
\begin{align*}
k_c(\gamma)  =
\begin{cases}
k_c^-(\gamma), &\text{if } \big|\mathbb{E}[\tau^\gamma] - \mathbb{E}[\tau_{k_c^-(\gamma)}]\big|\leq  \big|\mathbb{E}[\tau^\gamma] - \mathbb{E}[\tau_{k_c^+(\gamma)}]\big|,\\
k_c^+(\gamma), & \text{otherwise},
\end{cases}
\end{align*}
i.e., the value of $k$ for which $\mathbb{E}[\tau_k]$ is closest to $\mathbb{E}[\tau^\gamma]$.

\begin{Remark}{\bf [Activation rate]}\label{rem:act}
\rm{A key object in analysing $k_c(\gamma)$ is the activation rate at time $t$ for $k$ key nodes, given by
\begin{align*}
m_k(t) = \lim_{\Delta\downarrow 0} \frac{1}{\Delta}\,\P \big\{\tau_k \leq t+\Delta \mid \tau_k \geq t\big\}
= -\frac{1}{\P\{\tau_k\geq t\}}\,\frac{\mathrm{d}}{\mathrm{d}t}\P\{\tau_k\geq t\}.
\end{align*}
We will also look at $m_k(t)$ later. \hfill$\spadesuit$
}
\end{Remark}

Under a {\em mean-field assumption} and a {\em homogeneity assumption} (see below), we estimate $k_c(\gamma)$ for different choices of $\lambda(\cdot)$ and $\mu(\cdot)$ when the number $n$ of aggregated nodes is large. These two assumptions are reasonable due to the scale-free and disassortative nature of the signal network. The Pareto principle, popularised by Richard Koch \cite{Koch}, suggests that roughly $80\%$ of effects result from $20\%$ of causes. Therefore, we consider $\gamma = 0.4$ as a conservative estimate of the fraction of nodes that need to be active to trigger the entire network, \emph{regardless} of the specific nature of the signal network. The theory presented in Sections \ref{s.modass}--\ref{s.math} works for every choice of $\gamma \in (0,1)$, but for the numerics in Section \ref{s.numerics} we use the choice $\gamma = 0.4$.


\subsection{Mean-field assumption}
\label{ss.MF} 

The mean-field assumption is the approximation
\begin{align*}
a_{i}(t) \approx a(t) \approx \hat{a}(t)=\E[a(t)],
\end{align*}
in combination with the fact that rates for the transition to the active state or the passive state are independent of the individual nodes are given by
\begin{align*}
\Gamma_{+}(i,t) \approx \lambda(\hat{a}(t)), \qquad  \Gamma_{-}(i,t)\approx\mu(\hat{a}(t)).
\end{align*} 
We will refer to $\hat{a}(t)$ as the \emph{activation fraction} at time $t$.

In a large network, scale-freeness and disassortativity imply the presence of a fair number of high-degree nodes, which are connected to many low-degree nodes throughout the network, while their union still covers only a small fraction of the entire network. This leads to the nodes' near independence and results in them behaving approximately in the same way. It is, therefore, reasonable to apply the law of large numbers, which suggests that the average state of the nodes is a good approximation for the state of the entire network in support of the above mean-field assumption. 


\subsection{Homogeneity assumption}
\label{ss.H} 

For $i\in\{0,1,\ldots ,k\}$, we define 
\[
S_{i}=\Big\{(z_{1},\ldots, z_{n})\in\{0,1\}^{n}\colon\,\#\{1\leq j\leq k:z_{j}=1\}=i\Big\}.
\]
We assume that the states of the signal network in $S_{0},S_{1},\ldots, S_{k}$ are \emph{further aggregated} into $k+1$ distinct states. See Fig.\ \ref{fig:aggregated} for an illustration. In that case $(k-i)\lambda(\hat{a}(t))$ and $i\mu(\hat{a}(t))$ are the transition rates out of state $S_{i}$ and into state $S_{i+1}$, respectively, $S_{i-1}$. These transition rates are depicted in Figure \ref{fig:aggregated}. In fact, if $Q(t)=(q_{i,j}(t))_{0\leq i,j\leq k-1}$ denotes the infinitesimal generator (transition rate matrix) of the inhomogeneous continuous-time Markov chain ${\bf X}$ at time $t$, then
\begin{align*}
q_{i,j}(t)&=\lim_{\Delta\downarrow 0}\frac{\P\big\{\X (t+\Delta)\in S_{j}\mid \X (t)\in S_{i}\big\}-\delta_{i,j}}{\Delta} 
\\[1mm]&= 
\begin{cases}
(k - i)\lambda(\hat{a}(t)), & \text{if }  0 \leq i \leq k-1,\ j = i+1, \\
i\mu(\hat{a}(t)), & \text{if } 1 \leq i \leq k ,\ j = i-1, \\
-[(k - i)\lambda(\hat{a}(t)) + i\mu(\hat{a}(t))], & \text{if } i = j, \\
0, & \text{otherwise},
\end{cases}
\end{align*}
where $\delta_{i,j}$ is the Kronecker delta (which is $1$ if $i=j$ and $0$ otherwise).

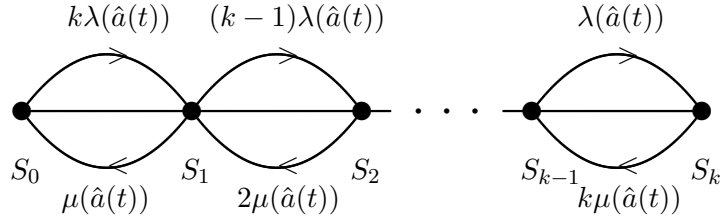
\begin{figure}[htbp]
\begin{center}
\setlength{\unitlength}{0.75cm}
\begin{picture}(12.5,4)(0,-2)
{\thicklines
  \qbezier(0,0)(3,0)(6.5,0)
  \put(7,0){\circle*{0.10}} 
  \put(7.5,0){\circle*{0.10}} 
  \put(8,0){\circle*{0.10}} 
  \qbezier(8.5,0)(10.5,0)(12,0)
  \qbezier(0,0)(1.5,2)(3,0)
  \qbezier(3,0)(1.5,-2)(0,0)
  \qbezier(3,0)(4.5,2)(6,0)
  \qbezier(6,0)(4.5,-2)(3,0)
  \qbezier(9,0)(10.5,2)(12,0)
  \qbezier(12,0)(10.5,-2)(9,0)
}
\put(0,0){\circle*{0.35}} 
\put(3,0){\circle*{0.35}} 
\put(6,0){\circle*{0.35}} 
\put(9,0){\circle*{0.35}} 
\put(12,0){\circle*{0.35}}
\put(-0.2,-1.2){$S_{0}$}
\put(2.8,-1.2){$S_{1}$}
\put(5.8,-1.2){$S_{2}$}
\put(8.8,-1.2){$S_{k-1}$}
\put(11.8,-1.2){$S_{k}$}
\put(0.8,1.5){$k\lambda(\hat{a}(t))$}
\put(0.7,-1.7){$\mu(\hat{a}(t))$}
\put(3.3,1.5){$(k-1) \lambda(\hat{a}(t))$}
\put(3.8,-1.7){$2\mu(\hat{a}(t))$}
\put(9.8,1.5){$\lambda(\hat{a}(t))$}
\put(9.8,-1.7){$k\mu(\hat{a}(t))$}
\put(1.5,-1.1){$<$}
\put(1.5,0.85){$>$}
\put(4.5,0.85){$>$}
\put(4.5,-1.1){$<$}
\put(10.5,0.85){$>$}
\put(10.5,-1.1){$<$}
\end{picture}
\end{center}
\caption{\small The transition rates between distinct states in the aggregated time-inhomogeneous Markov chain.}
\label{fig:aggregated}
\end{figure}


\section{Analytical Results}
\label{s.math}

In Section~\ref{ss.evolact}, with the help of the mean-field assumption and following a similar approach as in \cite{Aging}, we derive a mathematical expression for the activation fraction $\hat{a}(t)$ at time $t$ as the solution to an autonomous differential equation, and for the hitting time $\tau^\gamma$, representing the first time when the network is triggered. In Section~\ref{ss.expact}, with the help of the mean-field and homogeneity assumptions, we derive a mathematical expression for the expectation of the hitting time $\tau_k$, defined as the first time when all $k$ key nodes get active, and in Section~\ref{ss.actrat} for the activation rate of the $k$ key nodes. We use the symbol $\approx$ to emphasise that the expressions are obtained under the two assumptions. Our key results are Theorems~\ref{th:tauk} and \ref{th:Glaw} below. 


\subsection{Evolution of the activation fraction}
\label{ss.evolact}

\begin{Lemma}{\bf [Evolution of the activation fraction]}
\label{Lem1}
The activation fraction $\hat{a}(t)$ at time $t$ is the solution to an autonomous differential equation 
\begin{align}
\label{AP-Ahat}
\frac{\mathrm{d}}{\mathrm{d}t}\hat{a}(t)\approx(1-\hat{a}(t))\lambda(\hat{a}(t))-\hat{a}(t)\mu(\hat{a}(t)),\qquad \hat{a}(0)=0,
\end{align}
where $\lambda(\hat{a}(t))$ is the rate at time $t$ at which a node transitions from the passive state to the active state, and $\mu(\hat{a}(t))$ is the rate at time $t$ at which a node transitions from the active state to the passive state.
\end{Lemma}

\begin{proof}
Since all nodes are initially passive, we have $\hat{a}(0)=0$. Moreover, the rate at which one of the $n(1-\hat{a}(t))$ passive nodes at time $t$ becomes active equals $n(1-\hat{a}(t))\lambda(\hat{a}(t))$, while the rate at which one of the $n\hat{a}(t)$ active nodes at time $t$ becomes passive equals $n\hat{a}(t)\mu(\hat{a}(t))$. Therefore, $n(1-\hat{a}(t))\lambda(\hat{a}(t))-n\hat{a}(t)\mu(\hat{a}(t))$ describes how the number of active nodes evolves (increases or decreases) at time $t$. Divide this value by $n$ to get the fraction of active nodes at time $t$.
\end{proof}

\begin{Remark}{\bf [Link to the trigger time]}
\label{taualpha}
\rm{Under the mean-field assumption, $\tau^\gamma$ is deterministic and is equal to
\begin{align*}
\hat{\tau}^\gamma = \inf\big\{t\geq 0\colon\, \hat{a}(t)\geq \gamma\big\},
\end{align*}
i.e., the first time the activation fraction at time $t$ reaches the threshold level $\gamma$. Therefore, we analyse $\tau^\gamma$ by solving the equation in \eqref{AP-Ahat} and then looking at the first time $t$ at which $\hat{a}(t)\geq \gamma$. 
\hfill$\spadesuit$}
\end{Remark}


\subsection{Expected activation time of the key nodes}
\label{ss.expact}

\begin{Theorem}{\bf [Expected time for key nodes activation]}
\label{th:tauk}
For $k \geq 1$, 
\begin{align}
\label{tauk}
\E[\tau_{k}]\approx\int_{0}^{\infty}(1-p_{k}(t))\, \dd t,
\end{align}
where $p_{k}(t)$ is obtained from the following system of coupled ordinary differential equations:
\begin{equation}
\label{ODEs}
\begin{aligned}
\frac{\ddd}{\ddd t}\,p_{0}(t) &= p_{1}(t)\mu(\hat{a}(t))-p_{0}(t)k\lambda(\hat{a}(t)),\\
\frac{\ddd}{\ddd t}\,p_{j}(t) &= p_{j-1}(t)(k-j+1) \lambda(\hat{a}(t))+p_{j+1}(t) (j+1)\mu(\hat{a}(t))\\
&\qquad \qquad -p_{j}(t)\big[(k - j)\lambda(\hat{a}(t)) + j\mu(\hat{a}(t))\big],
\qquad\qquad\qquad\,\, 1\leq j\leq k-2,\\
\frac{\ddd}{\ddd t}\,p_{k-1}(t) &= p_{k-2}(t)\, 2\lambda(\hat{a}(t))-p_{k-1}(t)\big[\lambda(\hat{a}(t)) 
+ (k-1)\mu(\hat{a}(t))\big],
\qquad k>1,\\
\frac{\ddd}{\ddd t}\,p_{k}(t) &= p_{k-1}(t) \lambda(\hat{a}(t)),
\end{aligned}
\end{equation}
with the following initial conditions:
\begin{align*}
p_{0}(0)=1 \quad \text{ and } \quad p_{1}(0)=\cdots =p_{k}(0)=0.
\end{align*}
(Recall that $\lambda(\hat{a}(t))$ and $\mu(\hat{a}(t))$ are the rates at time $t$ at which a node transitions from passive to active, respectively, from active to passive, and $\hat{a}(t)$ is the activation fraction at time $t$.) Here, for $t\geq 0$ and $j=0,1,\ldots ,k$,
\begin{align*}
p_{j}(t)\in [0,1] \quad\text{ and }\quad \sum_{j=0}^{k}p_{j}(t)=1.
\end{align*}
\end{Theorem}
\begin{proof}
Let $\tilde{\mathbf{X}}$ be a time-inhomogeneous Markov chain on the state space $\{S_0, S_1, \ldots, S_k\}$, with transition rates identical to those of $\mathbf{X}$, except that the transition from $S_k$ to $S_{k-1}$ is removed and $S_k$ is considered as an absorbing state. Let $p_{j}(t)=\P\{\tilde{\X}(t)\in S_{j}\}$ for $0\leq j\leq k$. Since all nodes are initially passive, we have $p_{0}(0)=1$ and $p_{1}(0)=\cdots =p_{k}(0)=0$. 

For $\Delta \downarrow 0$ and $0\leq j\leq k-1$ we have
\begin{align*}
p_{j}(t+\Delta) = \sum_{i=0}^{k}p_{i}(t)\, \P\big\{\tilde{\X} (t+\Delta)\in S_{j}\mid \tilde{\X} (t)\in S_{i}\big\}&
= \sum_{i=0}^{k-1}p_{i}(t)\, \big( q_{i,j}(t)\Delta + \delta_{i,j} +o(\Delta)\big)\\&
=p_{j}(t)+\sum_{i=0}^{k-1}p_{i}(t)\, q_{i,j}(t)\Delta +o(\Delta).
\end{align*}
Similarly, for $\Delta \downarrow 0$ we have
\begin{align*}
p_{k}(t+\Delta) &= \sum_{i=0}^{k}p_{i}(t)\, \P\big\{\tilde{\X} (t+\Delta)\in S_{k}\mid \tilde{\X} (t)\in S_{i}\big\}
= p_{k}(t)+ \sum_{i=0}^{k-1}p_{i}(t)\, q_{i,k}(t)\Delta +o(\Delta).
\end{align*}
This implies that, for $0\leq j\leq k$,
\begin{align*}
\frac{\ddd}{\ddd t}\,p_{j}(t)=\sum_{i=0}^{k-1}p_{i}(t)\, q_{i,j}(t).
\end{align*}
Hence we have
\begin{align*}
&\frac{\ddd}{\ddd t}\,p_{0}(t)=p_{0}(t)\, q_{0,0}(t)+p_{1}(t)\, q_{1,0}(t)=p_{1}(t)\mu(\hat{a}(t))-p_{0}(t)k\lambda(\hat{a}(t)),
\\&\frac{\ddd}{\ddd t}\,p_{k}(t)=p_{k-1}(t)\, q_{k-1,k}(t)=p_{k-1}(t) \lambda(\hat{a}(t)),
\end{align*}
and, for $k>1$,
\begin{align*}
\frac{\ddd}{\ddd t}\,p_{k-1}(t)&=p_{k-2}(t)\, q_{k-2,k-1}(t)+p_{k-1}(t)\, q_{k-1,k-1}(t)\\&
=p_{k-2}(t)\, 2\lambda(\hat{a}(t))-p_{k-1}(t)\big[\lambda(\hat{a}(t)) + (k-1)\mu(\hat{a}(t))\big],
\end{align*}
and, when $k>2$, for $1\leq j\leq k-2$,
\begin{align*}
\frac{\ddd}{\ddd t}\,p_{j}(t)&=p_{j-1}(t)\, q_{j-1,j}(t)+p_{j}(t)\, q_{j,j}(t)+p_{j+1}(t)\, q_{j+1,j}(t)\\&
=p_{j-1}(t)(k-j+1) \lambda(\hat{a}(t))+p_{j+1}(t) (j+1)\mu(\hat{a}(t))-p_{j}(t)\big[(k - j)\lambda(\hat{a}(t)) + j\mu(\hat{a}(t))\big] .
\end{align*}

Finally, 
\begin{align*}
\E[\tau_{k}]=\int_{0}^{\infty}\P\{\tau_{k}>t\}\, \dd t\approx\int_{0}^{\infty}\P\{\tilde{\X}(t)\notin S_{k}\}\, \dd t=\int_{0}^{\infty}(1-p_{k}(t))\, \dd t.
\end{align*}
\end{proof}


\subsection{Activation rates for the key nodes}
\label{ss.actrat}

\begin{Lemma}{\bf [Representation of the activation rate in terms of the pre-trigger probability]}
\label{lem:mt}
For every $t \geq 0$,
\begin{equation*}
\label{eq:mortalitysplit}
\begin{aligned}
m_k(t) &\approx \lambda(\hat{a}(t))\, \mathbb{P}\big\{{\bf X}(t) \in S_{k-1} \mid {\bf X}(u) \not \in S_{k}, 0 \leq u \leq t \big\}.
\end{aligned}
\end{equation*}
(Recall that $\lambda(\hat{a}(t))$ is the rate at time $t$ at which a node transitions from passive to active, and $\hat{a}(t)$ is the activation fraction at time $t$.)
\end{Lemma}

\begin{proof}
In a short time length, ${\bf X}$ can only reach $S_{k}$ by going from $S_{k-1}$ to $ S_{k}$ in a single jump. Hence, taking
\[
\begin{aligned}
A_t &= \{{\bf X}(u) \not \in S_{k}, 0 \leq u \leq t\}=\{\tau_k \geq t\},\\
B_t &= \{{\bf X}(t) \in S_{k-1}\},\\
C_{t,\Delta} &=  \{\exists\,0 \leq u \leq \Delta\colon\, {\bf X}(t+u) \in S_{k}\},
\end{aligned}
\]
we have, for $\Delta \downarrow 0$,
\[
\begin{aligned}
\mathbb{P}\big\{\tau_k \leq t+\Delta \mid  \tau_k \geq t\big\}
= \frac{\mathbb{P}\{A_t \cap B_t \cap C_{t,\Delta} \}}{\mathbb{P}\{A_t\}}+o(\Delta)&
= \mathbb{P}\{C_{t,\Delta} \mid A_t\cap B_t\}\,\mathbb{P}\{B_t \mid A_t\}+o(\Delta)\\&
= \mathbb{P}\{C_{t,\Delta} \mid B_t\}\,\mathbb{P}\{B_t \mid A_t\}+o(\Delta),
\end{aligned}
\]
where the last equality follows from the Markov property at time $t$. Noting that the rate of jumping from $S_{k-1}$ to $S_k$ is $\lambda(\hat{a}(t))$, we have
\begin{align*}
\mathbb{P}\{C_{t,\Delta} \mid B_t\}\approx\lambda(\hat{a}(t))\,\Delta + o(\Delta),\qquad \Delta \downarrow 0,
\end{align*}
and hence
\[
\begin{aligned}
m_k(t) =\lim_{\Delta \downarrow 0} \frac{1}{\Delta}\mathbb{P}\big\{\tau_k \leq t+\Delta \mid  \tau_k \geq t\big\}
 \approx \lambda(\hat{a}(t))\,\mathbb{P}\{B_t \mid A_t\}.
\end{aligned}
\]
\end{proof}

\begin{Lemma}{\bf [Computation of the pre-trigger probability]}
\label{lem:a2}
For every $k \geq 1$,
\[
\mathbb{P}\big\{{\bf X}(t) \in S_{k-1} \mid {\bf X}(u) \not \in S_{k}, 0 \leq u \leq t \big\}
\approx \frac{p_{k-1}(t)}{1-p_{k}(t)},
\]
where $p_{k-1}(t)$ and $p_{k}(t)$ are defined as in Theorem \ref{th:tauk}, i.e., as the solutions of the ordinary differential equations in \eqref{ODEs}.
\end{Lemma}

\begin{proof}
Define $\tilde{\mathbf{X}}$ as in the proof of Theorem \ref{th:tauk}. Also, let $p_{j}(t)=\P\{\tilde{\X}(t)\in S_{j}\}$ for $0\leq j\leq k$. Then
\begin{align*}
\mathbb{P}\big\{{\bf X}(t) \in S_{k-1} \mid {\bf X}(u) \not \in S_{k}, 0 \leq u \leq t \big\}
&=\dfrac{\mathbb{P}\big\{\{{\bf X}(t) \in S_{k-1}\}\cap\{{\bf X}(u) \not \in S_{k}, 0 \leq u \leq t \}\big\}}{\mathbb{P}\big\{{\bf X}(u) \not \in S_{k}, 0 \leq u \leq t \big\}}\\&
\approx\dfrac{\mathbb{P}\big\{\tilde{\X}(t) \in S_{k-1}\big\}}{\mathbb{P}\big\{\tilde{\X}(t) \not \in S_{k} \big\}}= \frac{p_{k-1}(t)}{1-p_{k}(t)}.
\end{align*}
\end{proof}

\begin{Theorem}{\bf [Activation rate for key nodes]}
\label{th:Glaw}
For every $k \geq 1$,
\begin{equation}
\label{eq:GMlaw}
m_{k}(t) \approx \lambda(\hat{a}(t))\,\frac{p_{k-1}(t)}{1-p_{k}(t)},
\end{equation}
where $p_{k-1}(t)$ and $p_{k}(t)$ are defined as in Theorem \ref{th:tauk}, i.e., as the solutions of the ordinary differential equations in \eqref{ODEs}.
\end{Theorem}

\begin{proof}
Combine Lemmas~\ref{lem:mt}--\ref{lem:a2} to obtain \eqref{eq:GMlaw}.
\end{proof}


\section{Numerical Results}
\label{s.numerics}

In Section~\ref{ss.optch}, we investigate the optimal number of key nodes $k_c(\gamma)$ for five specific choices of the rate functions $\lambda(\cdot)$ and $\mu(\cdot)$ governing the transitions to the active and to the passive state, respectively. In Section~\ref{ss.optag}, we do the same for the aging network and the rate functions simulated in \cite{MRFR2017}. The numerics is based on the approach mentioned in Remark~\ref{taualpha} and the choice $\gamma=0.4$ mentioned below Remark~\ref{rem:act}.


\subsection{Optimal number of key nodes for various rate functions}
\label{ss.optch}

We assume that $\lambda(a) < \mu(a)$ for all $a\in (0,1]$, which ensures that the system does not exhibit unbounded growth in activation and, instead, tends to a stable and predominantly passive state.
We explore several rate functions (constant, exponential, power-law, logarithmic) and evaluate $\mathbb{E}[\tau_k]$ (see~\eqref{tauk}) for $k=1$ to $6$. The corresponding results are presented in Table \ref{tab:merged}. For the constant rate functions $\lambda(a) = 0.45$ and $\mu(a) = 0.65$, the optimal number of key nodes appears to be $k_c(\gamma) = 2$. The same optimal value, $k_c(\gamma) = 2$, arises in both the exponential-constant ($\lambda(a) = \mathrm{e}^{-a}$, $\mu(a) = 1$) and constant-exponential ($\lambda(a) = 1$, $\mu(a) = \mathrm{e}^a$) scenarios. For the power-constant rate functions $\lambda(a) = (a+1)^{-1}$ and $\mu(a) = 1$, the numerical results suggest $k_c(\gamma) = 1$; however, due to the proximity of $\hat{\tau}^{\gamma}$ to the estimate of $\mathbb{E}[\tau_2]$, and the potential influence of numerical error, the value $k = 2$ may still be regarded as a reasonable estimate. Lastly, for the logarithmic rate functions $\lambda(a) = (\log(3+a))^{-1}$ and $\mu(a) = \log(3+a)$, the numerical findings again support $k_c(\gamma) = 2$ as the optimal number of key nodes.

\vspace*{3mm}
\setlength{\tabcolsep}{10pt}
\begin{table}[htbp]
\caption{Numerical estimates of the expected hitting time $\mathbb{E}[\tau_k]$ and the trigger time $\hat{\tau}^{\gamma}$  for various choices of the rate functions and $\gamma=0.4$. The optimal number of key nodes $k_c(\gamma)$, for which $\mathbb{E}[\tau_k]$ is closest to $\hat{\tau}^{\gamma}$, corresponds to the $k$-values indicated in boldface. Recall Remark~\ref{taualpha}.}
\vspace*{-3mm}
\label{tab:merged}
\centering
\small
\renewcommand{\arraystretch}{1.3}
\begin{tabular}{@{}cccccc@{}}
\toprule
\textbf{$k$} &
\begin{tabular}{@{}c@{}}$\lambda(a) = 0.45$ \\ $\mu(a) = 0.65$\end{tabular} &
\begin{tabular}{@{}c@{}}$\lambda(a) = \eee^{-a}$ \\ $\mu(a) = 1$\end{tabular} &
\begin{tabular}{@{}c@{}}$\lambda(a) = 1$ \\ $\mu(a) = \eee^a$\end{tabular} &
\begin{tabular}{@{}c@{}}$\lambda(a) = (a+1)^{-1}$ \\ $\mu(a) = 1$\end{tabular} &
\begin{tabular}{@{}c@{}}$\lambda(a) = (\log(3+a))^{-1}$ \\ $\mu(a) = \log(3+a)$\end{tabular} \\
\midrule
1                              & 1.37 & 0.79 & 0.62 & \textbf{0.78}  & 0.72 \\
2                           & \textbf{4.94} & \textbf{3.14} & \textbf{2.22} & 2.95 & \textbf{2.69} \\
3                              & 9.36 & 6.19 & 4.27 & 5.66  & 5.22 \\
4                              & 17.48 & 11.94 & 8.12 & 10.61 & 9.97 \\
5                              & 33.35 & 23.45 & 15.84 & 20.19 & 19.49 \\
6                              & 65.08 & 47.34 & 31.88 & 39.49 & 39.30 \\
\midrule
$\hat{\tau}^{\gamma}$    & 3.46 & 2.79 & 2.00 & 1.61 & 3.11 \\
\bottomrule
\end{tabular}
\end{table}


\subsection{Optimal number of key nodes in an aging network}
\label{ss.optag}

Consider the rate functions chosen in \cite{MRFR2017} (and analysed in \cite{Aging}),
\begin{equation}
\label{rate-aging}
\lambda(a) = \Gamma_0\,\eee^{r_+ a},  \qquad \mu(a) = \frac{\Gamma_0}{R} \,\eee^{-r_- a}, \qquad a \in [0,1],
\end{equation}
where the parameters are chosen to be $r_+ = 10.27$, $r_- = 6.5$, $R = 1.5$ and $\Gamma_0 = 0.00113$.

\setlength{\tabcolsep}{8pt}
\begin{table}[htbp]
\caption{Numerical estimates of the expected hitting times $\mathbb{E}[\tau_k]$ for the rate functions in \eqref{rate-aging}.}
\vspace*{-3mm}
\label{tab:etauk}
\centering
\small
\renewcommand{\arraystretch}{1.5}
\begin{tabular}{lccccccc}
\toprule
\textbf{$k$} & 1 & 2 & 3 & 4 & 5 & 6 & 7 \\
\midrule
\textbf{$\mathbb{E}[\tau_k]$} & 86.94 & 95.49 & 97.14 & 97.62 & 97.80 & 97.89 & 97.93 \\
\bottomrule
\end{tabular}
\end{table}


\setlength{\tabcolsep}{8pt}
\begin{table}[htbp]
\caption{Numerical estimates of $\hat{\tau}^{\gamma}$ for various values of $\gamma$ for the rate functions in \eqref{rate-aging}.}
\vspace*{-3mm}
\label{tab:taualpha}
\centering
\small
\renewcommand{\arraystretch}{1.5}
\begin{tabular}{lccccc}
\toprule
\textbf{$\gamma$} & 0.2 & 0.3 & 0.4 & 0.5 & 0.6 \\
\midrule
\textbf{$\hat{\tau}^{\gamma}$} & 81.79 & 91.19 & 95.05 & 96.69 & 97.40 \\
\bottomrule
\end{tabular}
\end{table}

\vspace*{3mm}
A comparison of Tables~\ref{tab:etauk}--\ref{tab:taualpha} shows that $k_c(\gamma) = 1$ for $\gamma = 0.2, 0.3$, $k_c(\gamma) = 2$ for $\gamma = 0.4$, and $k_c(\gamma) = 3$ for $\gamma = 0.5$. For $\gamma = 0.6$, although the numerical estimates suggest that $k_c(\gamma) = 4$, the small difference of $\hat{\tau}^{0.6}$ from the estimate of $\mathbb{E}[\tau_3]$, along with the potential impact of numerical errors, indicates that $k_c(\gamma) = 3$ is also a reasonable and efficient choice. Note that for the parameter values used in~\eqref{rate-aging}, the threshold $\gamma = 0.4$ is particularly appropriate, as death can occur when approximately half of the aggregated nodes are damaged. Therefore, the optimal number of key nodes is $k_c(\gamma) = 2$. Even for a conservative choice of $\gamma \in [0.4, 0.6]$, it is reasonable to take $k_c(\gamma) \in \{2, 3\}$. This provides a mathematical justification for using two mortality nodes in a health network, as adopted in~\cite{MRFR2017, Aging}.

We also note that the rate functions in \eqref{rate-aging} fail to satisfy the inequality $\lambda(a)<\mu(a)$ for all $a \in (0,1]$, which makes it somewhat different from the five choices used in Section~\ref{ss.optch}. 

\subsection{Sensitivity to the choice of transition rate functions.}
To assess the robustness of our results, we compared several qualitatively different yet theoretically plausible functions for the transition rates $\lambda(\cdot)$ and $\mu(\cdot)$ in Sections \ref{ss.optch} and \ref{ss.optag}: constant, exponential, power-law, logarithmic, and the form proposed in \cite{MRFR2017}. These choices cover a wide range of behaviours, from constant rates to rates that fastly or slowly decrease or increase in time, for both the transition rate $\lambda(\cdot)$ from the passive state to the active state and the transition rate $\mu(\cdot)$ from the active state to the passive state. Interestingly, we found that, even though the functions chosen are very different, the threshold values $k_c(\gamma)$ for $\gamma=0.4$ were almost the same for all choices. This means that, apparently, our results hardly depend on the precise form of $\lambda(\cdot)$ and $\mu(\cdot)$, which observation is a type of sensitivity analysis.


\section{Conclusion}

We have shown that, for large scale-free and disassortative signal networks, only a very small number of nodes need to be monitored to get a fair impression of the overall state of the network. This dramatic form of \emph{data aggregation} is important because for large networks it is difficult, if not impossible, to monitor the overall state. We have further shown that the reduction is \emph{universal} and is not related to the underlying structure of the network. Three examples serve as an illustration: brain networks, aging networks, and friendship networks. Based on two assumptions -- mean-field and homogeneity -- we have shown how the optimal number of key nodes can be identified analytically and computed numerically for different activation and deactivation rates.  

Beyond their theoretical importance, our findings have practical implications. For brain networks, monitoring a small set of hubs can help to obtain a faster and less invasive diagnosis of neurological disorders. In aging networks, it can help to estimate life expectancy by tracking the health of a few critical organs or health indicators. In social networks, it can help to guide efficient information spread or contain the spread of harmful rumours.



\section*{Data Availability}
No external datasets were used in this study. All data generated during this study are included in this published article. Numerical results were obtained through solving equations using R software.

\section*{Authors Contribution Statement} IL and BH conceptualised the design of the study, FdH and AP performed the mathematical analysis, and all authors jointly wrote the manuscript text.

\section*{Additional Information}

We have no competing conflicts of interest to disclose.


\end{document}